\documentclass[letterpaper,12pt]{amsart}
\usepackage{amsmath,amssymb,amsxtra,amsthm}
\theoremstyle{plain}
\newtheorem{theorem}{Theorem}[section]

\newtheorem{lemma}[theorem]{Lemma}
\newtheorem{proposition}[theorem]{Proposition}

\theoremstyle{definition}

\begin{document}
\title{A pattern sequence approach to Stern's sequence}
\author{Michael Coons}
\author{Jeffrey Shallit}
\address{University of Waterloo, Dept.~of Pure Mathematics, Waterloo, Ontario, N2L 3G1}
\email{mcoons@math.uwaterloo.ca}
\address{University of Waterloo, School of Computer Science, Waterloo, Ontario, N2L 3G1}
\email{shallit@cs.uwaterloo.ca}
\thanks{The research of M.~Coons is supported by a Fields--Ontario Fellowship and NSERC, and the research of J.~Shallit is supported by NSERC} 
\date{\today}

%%%%%%%%%%%%%%%%%%%%%%%%%%%%%%%%%%%%%%%%%%%%%%%%%%%%%%%%
%%%%%%%%%%%%%%%%%%%%%%%%%%%%%%%%%%%%%%%%%%%%%%%%%%%%%%%%

\begin{abstract} Let $w \in 1\{0,1\}^*$ and let $a_w (n)$ be the number
of occurrences of the word $w$ in the binary expansion of
$n$. Let $\{s(n)\}_{n\geqslant 0}$ denote the Stern
sequence, defined by $s(0)=0$, $s(1)=1$, and for $n\geqslant 1$
$$s(2n)=s(n),\quad{\rm and}\quad s(2n+1)=s(n)+s(n+1).$$ In this note,
we show that $$s(n) = a_1 (n) +     \sum_{w \in 1\{0,1\}^*} s( [ \overline{w} ]_2 )
a_{w1} (n) $$ where $\overline{w}$ denotes the complement of $w$
(obtained by sending $0\mapsto 1$ and $1\mapsto0$) and $[w]_2$ denotes
the integer specified by the word $w\in \{0,1\}^*$ interpreted in base
$2$.
\end{abstract}

\maketitle

\section{Introduction}

For $w\in1\{0,1\}^*$ let 
$a_w(n)$ denote the number of (possibly 
overlapping) occurrences of
the word $w$ in the binary expansion of $n$.

For $w \in \{0,1\}^*$ let

\begin{itemize}

\item $\overline{w}$ denote the complement of the word $w$ 
obtained by sending $0\mapsto 1$ and
$1\mapsto0$, and

\item $[w]_2$ denote the integer specified by the word $w$ interpreted in base $2$.
\end{itemize}

Morton and Mourant \cite{MM} proved that every sequence of
real numbers
$\lbrace S(n) \rbrace_{n \geqslant 0}$ with $S(0) = 0$
has a unique {\it pattern sequence expansion}
$$ S(n) = \sum_{w \in 1\lbrace 0,1 \rbrace^*} \hat{S} ([w]_2) a_w(n) $$
where $\hat{S} : \mathbb{N} \to \mathbb{R}$; here we have used the definition
$\mathbb{N} := \lbrace 1,2,3,\ldots \rbrace$. See also \cite{AMS,AS} and \cite[Thm.\ 3.3.4]{AS2}.

The {\em Stern sequence} (also called {\em Stern's diatomic sequence};
sequence
A002487 in Sloane's {\it Encyclopedia}) is defined by the recurrence relations
$s(0)=0$, $s(1)=1$, and, for $n \geqslant 1$, by 
$$s(2n)=s(n),\quad{\rm and}\quad s(2n+1)=s(n)+s(n+1).$$
This famous sequence has many interesting properties; for example,
see the recent survey of Northshield \cite{N}.

Applying Morton and Mourant's theorem to the Stern sequence, we obtain
$$ s(n) = a_1 (n) + a_{101}(n) + 2 a_{1001}(n) + a_{1011} (n) +
a_{1101} (n) + 3 a_{10001} (n) + \cdots $$
In this note, we give the closed form for this expansion,
as follows:

\begin{theorem}\label{main} For all $n\geqslant 0$, we have $$s(n) = a_1 (n) +     \sum_{w \in 1\{0,1\}^*} s( [ \overline{w} ]_2 ) a_{w1} (n) .$$
\end{theorem} 

Note that if $w \in 1\lbrace 0,1 \rbrace^*$, then 
$[\overline{w}]_2 < [w]_2$, so that our theorem could also serve as an 
alternate definition of $s(n)$.

Our result can be contrasted with an unpublished result of Calkin and Wilf \cite[Thm.\ 5]{CW}, which was recently rediscovered by Bacher \cite[Prop.\ 1.1]{B} (see also Finch \cite[p.~148]{F}):

\begin{proposition}[Calkin and Wilf] 
The $n$th Stern value $s(n)$ is equal to
the number of subsequences of the form
$1,101,10101,\ldots=\{1(01)^*\}$ in the binary expansion of $n$.
\end{proposition}

Here by ``subsequence'' we mean a not-necessarily-contiguous subsequence.

%%%%%%%%%%%%%%%%%%%%%%%%%%%%%%%%%%%%%%%%%%%%%%%%%%%%%%%%
\section{Needed lemmas}
%%%%%%%%%%%%%%%%%%%%%%%%%%%%%%%%%%%%%%%%%%%%%%%%%%%%%%%%

We will need the following two lemmas, the first of which is quite classical and follows directly from the work of Stern \cite{S} (and is easily proven using induction on $k$).

\begin{lemma}[Stern]\label{s2k} 
Let $k$ and $n$ be nonnegative integers. If $n\leqslant 2^k$ then
$$s(2^k+n)=s(2^k-n)+s(n).$$
\end{lemma}

\begin{lemma}\label{s2} 
Let $n\geqslant 1$ and write $n=\sum_{\ell=0}^k 2^{i_\ell}$ where
$0\leqslant i_0<\cdots<i_k.$ Then
$$s(n+1)=1+\sum_{m=0}^k
s\left(2^{i_m+1}-\sum_{\ell=0}^m 2^{i_\ell}-1\right).$$
\end{lemma}

\begin{proof} 
Let $n\geqslant 1$ and write $n=\sum_{\ell=0}^k 2^{i_\ell}$ where
$0\leqslant i_0<\cdots<i_k$. Note that 
$$\sum_{m=0}^k
s\left(2^{i_m+1}-\sum_{\ell=0}^m 2^{i_\ell}-1\right)=\sum_{m=0}^k
s\left(2^{i_m}-\sum_{\ell=0}^{m-1} 2^{i_\ell}-1\right),$$ 
where we use
the ordinary convention that the empty sum equals $0$.

Since $$\sum_{\ell=0}^{m-1} 2^{i_\ell}+1\leqslant 2^{i_m},$$ 
we can apply Lemma~\ref{s2k} to give 
\begin{align*}\sum_{m=0}^k s\left(2^{i_m}-\sum_{\ell=0}^{m-1} 2^{i_\ell}-1\right) &=\sum_{m=0}^k \left(s\left(2^{i_m}+\sum_{\ell=0}^{m-1} 2^{i_\ell}+1\right)-s\left(\sum_{\ell=0}^{m-1} 2^{i_\ell}+1\right)\right)\\
&=\sum_{m=0}^k \left(s\left(\sum_{\ell=0}^{m} 2^{i_\ell}+1\right)-s\left(\sum_{\ell=0}^{m-1} 2^{i_\ell}+1\right)\right)\\
&=s\left(\sum_{\ell=0}^{k} 2^{i_\ell}+1\right)-s(1)\\
&=s(n+1)-1.
\end{align*}
A minor rearrangement gives the desired result.
\end{proof}

%%%%%%%%%%%%%%%%%%%%%%%%%%%%%%%%%%%%%%%%%%%%%%%%%%%%%%%%
\section{Proof of the theorem}
%%%%%%%%%%%%%%%%%%%%%%%%%%%%%%%%%%%%%%%%%%%%%%%%%%%%%%%%

Let $k\in\mathbb{N}$ and let $(k)_2$ denote the unique word $w \in
1\{0,1\}^*$ for which $[w]_2=k$; that is, $(k)_2$ is the canonical
base-$2$ representation of the integer $k$. We say that $w\in\{0,1\}^*$
is a {\em suffix} of $(k)_2$ if there exists $v\in\{0,1\}^*$
(possibly empty) such that $vw=(k)_2$. As usual, let $|w|$ denote the
length of the word $w \in \{0,1\}^*$.

\begin{proof}[Proof of Theorem~\ref{main}]
Define the sequence $\{f(n)\}_{n\geqslant 0}$ by
$$f(n):= a_1 (n) +     \sum_{w \in 1\{0,1\}^*} s( [ \overline{w} ]_2 )
a_{w1} (n).$$ Thus to prove the theorem, it is enough to show that
$f(n)=s(n)$ for all $n$.  Our proof is by induction on $n$.

For $n=0$, the
sum over $w \in 1\{0,1\}^*$ is $0$, and so we have immediately that
$$f(0)=a_1(0)=0=s(0).$$ For $n=1$, again we have that the sum over $w
\in 1\{0,1\}^*$ is $0$, and so $$f(1)=a_1(1)=1=s(1).$$

For the induction step, we will need to obtain some recursions for $f(n)$.
For even integers, we have 
$a_1(2n)=a_1(n)$ since $(2n)_2=(n)_20.$ For this same reason, those $w
\in 1\{0,1\}^*$ with $a_{w1}(2n)\neq 0$ are precisely those for which
$a_{w1}(n)\neq 0$, and so $a_{w1}(2n)=a_{w1}(n)$. Thus $$\sum_{w \in
1\{0,1\}^*} s( [ \overline{w} ]_2 ) a_{w1} (2n)= \sum_{w \in 1\{0,1\}^*} s( [
\overline{w} ]_2
) a_{w1} (n),$$ and so $$f(2n)=f(n).$$

Getting to our last relationship is a bit more complicated. The last
bit in $(2n+1)_2$ is equal to $1$. Recalling that those $w \in
1\{0,1\}^*$ with $a_{w1}(2n)\neq 0$ are precisely those for which
$a_{w1}(n)\neq 0$ and that $a_{w1}(2n)=a_{w1}(n)$, we have
\begin{equation}
\label{ss}
\sum_{w \in 1\{0,1\}^*} s( [ \overline{w} ]_2
) a_{w1} (2n+1)=\sum_{w \in 1\{0,1\}^*} s( [ \overline{w} ]_2 ) a_{w1}
(n)+\hspace{-.7cm}\sum_{\substack{w \in 1\{0,1\}^*1\\ w\ {\rm
is\ a\ suffix\ of}\ (2n+1)_2}} \hspace{-.7cm} s( [ \overline{w} ]_2
).
\end{equation}
Note that it is immediate that $a_1(2n+1)=a_1(n)+1$.
Combining this with \eqref{ss} gives that
\begin{equation}\label{ss1}f(2n+1)=f(n)+1+\hspace{-.7cm}\sum_{\substack{w
\in 1\{0,1\}^*\\ w1\ {\rm is\ a\ suffix\ of}\ (2n+1)_2}} \hspace{-.7cm}
s( [ \overline{w} ]_2 ).\end{equation} We have that $$\big\{w \in
1\{0,1\}^*: w1\ {\rm is\ a\ suffix\ of}\ (2n+1)_2\big\}=\big\{w \in
1\{0,1\}^*: w\ {\rm is\ a\ suffix\ of}\ (n)_2\big\}.$$

Now suppose that $n\geqslant 1$ and write $n=\sum_{\ell=0}^k
2^{i_\ell}$ where $0\leqslant i_0<i_1<\cdots<i_k.$ Noting that for all
$w \in 1\{0,1\}^*$, we get
$$[ \overline{w} ]_2=2^{|w|}-[w]_2-1.$$
This gives 
\begin{align*}\sum_{\substack{w \in 1\{0,1\}^*\\ w1\ {\rm
is\ a\ suffix\ of}\ (2n+1)_2}} \hspace{-.9cm} s( [ \overline{w} ]_2 )
=\hspace{-.3cm}\sum_{\substack{w \in 1\{0,1\}^*\\ w\ {\rm is\ a\ suffix\ of}\ (n)_2}} \hspace{-.9cm} s(2^{|w|}- [w]_2 -1)=\sum_{m=0}^k s\left(2^{i_m}-\sum_{\ell=0}^{m-1} 2^{i_\ell}-1\right). 
\end{align*}
Applying the previous equality and Lemma~\ref{s2} to \eqref{ss1} gives the
equality 
$$f(2n+1)=f(n)+s(n+1)\qquad (n\geqslant 1).$$ Recall that we have already shown that $f(0)=s(0)$, $f(1)=s(1)$ and that $f(2n)=f(n)$. 

We can now apply induction directly. Suppose that $f(j)=s(j)$ for all
$j$  such that $0\leqslant j<n$ and consider $f(n)$. If $n$ is even,
write $n=2m$ and note that $m<n$. Then
$$f(n)=f(2m)=f(m)=s(m)=s(2m)=s(n).$$ If $n\geqslant 3$ is odd, write
$n=2m+1$ and note that $1\leqslant m<n$. Then
$$f(n)=f(2m+1)=f(m)+s(m+1)=s(m)+s(m+1)=s(2m+1)=s(n).$$ 
Thus $f(n)=s(n)$
for all $n\geqslant 0$ and the theorem is proved.
\end{proof}

%%%%%%%%%%%%%%%%%%%%%%%%%%%%%%%%%%%%%%%%%%%%%%%%%%%%%%%%
%%%%%%%%%%%%%%%%%%%%%%%%%%%%%%%%%%%%%%%%%%%%%%%%%%%%%%%%

\end{document}